\documentclass{birkmult}
\usepackage[english]{babel} 
\usepackage{exscale}   
\usepackage{amsmath}
\usepackage{amsfonts}
\usepackage{amssymb}
\usepackage{amsxtra}
\usepackage{pstricks}
\usepackage{stmaryrd}
\usepackage{xspace}
\usepackage{epsfig}
\usepackage{mathrsfs}
 \newtheorem{thm}{Theorem}[section]
 \newtheorem{cor}[thm]{Corollary}
 \newtheorem{lem}[thm]{Lemma}
 \newtheorem{prop}[thm]{Proposition}
 \theoremstyle{definition}
 
 \theoremstyle{remark}

 \numberwithin{equation}{section}
 
\newcommand{\Dset}{\mathbb{D}}

\newcommand{\Nset}{\mathbb{N}}

\newcommand{\cB}{\ensuremath{{\mathcal B}}\xspace}         
\newcommand{\cD}{\ensuremath{{\mathcal D}}\xspace}         
\newcommand{\cH}{\ensuremath{{\mathcal H}}\xspace}         
\newcommand{\hH}{\ensuremath{{\hat{\mathcal H}}}\xspace}         
\newcommand{\cK}{\ensuremath{{\mathcal K}}\xspace}         
\newcommand{\cP}{\ensuremath{{\mathcal P}}\xspace}         
\newcommand{\cU}{\ensuremath{{\mathcal U}}\xspace}         
\newcommand{\cW}{\ensuremath{{\mathcal W}}\xspace}         
\newcommand{\cX}{\ensuremath{{\mathcal X}}\xspace}         
\newcommand{\cY}{\ensuremath{{\mathcal Y}}\xspace}         
\renewcommand{\tt}{\mathbf{t}}

\newcommand{\1}{\ensuremath{{\rm 1\kern-.25em l}}\xspace}  
 
\begin{document}
%
%
%
%
%
%
%
%
%
\title
 {Transfer Functions for Pairs of Wandering Subspaces}
\author{Rolf Gohm}

\address{%
Institute of Mathematics and Physics\\
Aberystwyth University\\
Aberystwyth SY23 3BZ\\
United Kingdom}

\email{rog@aber.ac.uk}




\begin{abstract}
To a pair of subspaces wandering with respect to a row isometry we associate a transfer function which in general is multi-Toeplitz and in interesting special cases is multi-analytic. Then we describe in an expository way how characteristic functions from operator theory as well as transfer functions from noncommutative Markov chains fit into this scheme.
\end{abstract}

\maketitle
\section*{Introduction}
It is evident to all workers in these fields that the relationship between operator theory and the theory of analytic functions is the source of many deep results. In recent work \cite{Go09} of the author a transfer function, which is in fact a multi-analytic operator, has been introduced in the context of noncommutative Markov chains. These can be thought of as toy models for interaction processes in quantum physics. The theory of multi-analytic operators, pioneered by Popescu \cite{Po89a,Po89b} and others in the late 1980's, has developed into a very successful generalization of the relationship mentioned above. Hence it is
natural to expect that noncommutative Markov chains and their transfer functions open up a possibility to apply these tools in the study of models in quantum physics.

This paper is the result of an effort to discover the common geometric underpinning which ties together these at first sight rather different settings. It is found in the tree-like structure of wandering subspaces of row isometries, more precisely: the transfer function describes the relative position of two such trees. This is worked out in Section 1 below. One of the main results in Section 1 is a geometric characterization of pairs of subspaces with a multi-analytic operator as their transfer function. 

With this work done we are in a position to discuss the existing applications in a new light which highlights common features. In Section 2 we give, from this point of view, an expository treatment of characteristic functions, both the well-known characteristic function of a contraction in the sense of Sz.Nagy and Foias \cite{SF70} and the less well-known characteristic function of a lifting introduced by Dey and Gohm \cite{DG11}. 
In Section 3 we explain in the same short but expository style the transfer function of a noncommutative Markov chain from \cite{Go09} and sketch a generalization which is natural in the setting of this paper. We hope and expect that this presentation is helpful for operator theorists to find their way into an area of potentially interesting applications. 

\section{Pairs of subspaces}
Let \hH be a Hilbert space and $V = (V_1,\ldots,V_d)$ a row isometry on $\hH$. Recall that this means that the
$V_k: \hH \rightarrow \hH$
are isometries with orthogonal ranges.
Here $d \in \Nset$ and additionally we also include the possibility of a sequence $(V_1, V_2, \ldots)$ of such isometries, writing symbolically $d=\infty$ in this case. 

Let $F^+_d$ be the free semigroup with $d$ generators (which we denote $1,\ldots,d$). Its elements are (finite) words in the generators, including the empty word (which we denote by $0$). The binary operation is concatenation of words. Let $\alpha = \alpha_1 \ldots \alpha_r$, with the $\alpha_\ell \in \{1,\ldots,d\}$, be such a word. We denote by $|\alpha|=r$ the length of the word $\alpha$. Further we define
\[
V_\alpha := V_{\alpha_1} \ldots V_{\alpha_r}  
\]
($V_0$ is the identity operator). By $V^*_\alpha$ we mean
$(V_\alpha)^* = V^*_{\alpha_r} \ldots V^*_{\alpha_1}$. We refer to \cite{Po89a,Po89b,BV05a,BV05b,DG11} for further background about this type of multi-variable operator theory. 

We want to establish an efficient description of the relative position of pairs of subspaces and their translates under a row isometry $V = (V_1,\ldots,V_d)$ on $\hH$.
Suppose $\cU$ and $\cY$ are Hilbert spaces and $i_0: \cU \rightarrow \hH$ and $j_0: \cY \rightarrow \hH$ are isometric embeddings into $\hH$. 
Further we write $i_\omega := V_\omega i_0$ and $i_\omega(\cU) =: \cU_\omega$, similarly
$j_\sigma := V_\sigma j_0$ and $j_\sigma(\cY) =: \cY_\sigma$,
where $\omega, \sigma \in F^+_d$. 
To describe the relative position of $\cU_\omega$ and
$\cY_\sigma$ we consider the contraction
\[
K(\sigma, \omega) 
:= j^*_\sigma\, i_\omega \colon \cU \rightarrow \cY\,.
\]
Note that 
\[
j_\sigma \, K(\sigma, \omega) \, i^*_\omega \colon  
\hH \rightarrow \hH
\]
is nothing but the orthogonal projection onto $\cY_\sigma$ restricted to $\cU_\omega$. The embeddings introduced above allow us to represent these contractions for varying $\sigma$ and $\omega$ on common Hilbert spaces $\cU$ and $\cY$. 

\pagebreak

\begin{lem}\label{lem:toeplitz}
$K(\sigma, \omega)$ for varying $\sigma$ and $\omega$ is a {\it multi-Toeplitz kernel}, i.e.,
\[
K: F^+_d \times F^+_d \rightarrow {\cB}(\cU,\cY)
\]
such that
\[
K(\sigma, \omega) = 
\left\{ 
\begin{array}{cc}
K(\alpha, 0) & {\mbox if}\; \sigma = \omega \alpha \\ 
K(0, \alpha) & {\mbox if}\; \omega = \sigma \alpha \\
0 & {\mbox otherwise} \\
\end{array}
\right .
\]
\end{lem}

\begin{proof}
If $\sigma = \omega \alpha$ then
\[
K(\sigma, \omega) = j^*_\sigma\, i_\omega 
= j^*_0 V^*_{\omega \alpha} V_\omega i_0 
= j^*_0 V^*_\alpha V^*_\omega V_\omega i_0 
= j^*_0 V^*_\alpha i_0 = K(\alpha, 0)\,. 
\]
Similarly if $\omega = \sigma \alpha$ then
\[
K(\sigma, \omega) = j^*_0 V_\alpha i_0 = K(0, \alpha)\,.
\]
Otherwise the orthogonality of the ranges of the $V_k$ forces $K(\sigma, \omega)$ to be $0$.
\end{proof}

Multi-Toeplitz kernels, in the positive definite case, have been investigated by Popescu, cf. \cite{Po99}. For more recent developments see also \cite{BV05a,BV05b}. Our focus will be on the analytic case, see Theorem \ref{thm:transfer} below. 

Let us introduce further notation and terminology. We define
\[
\cU_+ := \overline{span}\, \{\cU_\alpha \colon \alpha \in F^+_d \}
\]
\[
\cH := \hH \ominus \cU_+
\]
$\cU_+$ is the smallest closed subspace invariant for all $V_k$
containing $\cU_0$, and $\cH$ is invariant for all $V^*_k$.

A subspace $\cW \subset \hH$ is called {\it wandering} if $V_\alpha \cW \perp V_\beta \cW$ for $\alpha \not= \beta$ ($\alpha, \beta \in F^+_d$). We suppose from now on that $\cU_0$ is wandering. Then
$\cU_+ = \bigoplus_{\alpha\in F^+_d} \cU_\alpha$ (orthogonal direct sum), $V_k \cH \subset \cH \oplus \cU_0$ for all $k=1,\ldots,d$ and $V^*_\alpha\, \cU_0 \subset \cH$ for all $\alpha \not=0$. 

We can identify the space $\cU_+$ with 
$\ell^2(F^+_d, \cU)$,
the $\cU$-valued square-summable functions on $F^+_d$,
in the natural way.
If $\cY_0$ is also wandering then we can associate a multi-Toeplitz operator
\[
M \colon \ell^2(F^+_d, \cU) \rightarrow \ell^2(F^+_d, \cY)
\]
with a matrix given by the multi-Toeplitz kernel $K$ from 
Lemma \ref{lem:toeplitz}. In fact, using the identifications of 
$\cU_+ = \bigoplus_{\alpha\in F^+_d} \cU_\alpha$
with $\ell^2(F^+_d, \cU)$ and of
$\cY_+ = \bigoplus_{\alpha\in F^+_d} \cY_\alpha$
with $\ell^2(F^+_d, \cY)$
we see that $M$ is nothing but the orthogonal projection
onto $\cY_+$ restricted to $\cU_+$. Hence $M$ is a contraction which describes the relative position of $\cU_+$ and $\cY_+$.

We are interested in the case where the multi-Toeplitz kernel $K$ (resp. the multi-Toeplitz operator $M$) is {\it multi-analytic}, i.e.,
$K(0, \alpha) = 0$ for all $\alpha \not=0$. 
We note that the notion of multi-analytic operators has been studied in great detail by Popescu, cf. for example \cite{Po95}.

The following theorem gives several characterizations of multi-analyticity in our setting. 
The notation $P_{\cX}$ for the orthogonal projection onto a subspace $\cX$ is used without further comments.

\begin{thm} \label{thm:transfer}
Suppose that $\cU_0$ is wandering for the row isometry $V$ on $\hH$ and
let $\cY_0$ be any subspace of $\hH$. Then the following assertions are equivalent: 
\begin{itemize}
\item[(1)]
$K$ is multi-analytic.
\item[(2)]
$\cU_0 \perp V^*_\alpha \cY_0\;$ 
for all $\alpha \not= 0$
\item[(3)]
$\cY_0 \subset \cH \oplus \cU_0$
\item[(4)]
$V^*_k \cY_0 \subset \cH$
for all $k= 1,\ldots,d$
\item[(5)]
$V^*_\alpha \cY_0 \subset \cH$
for all $\alpha \not= 0$
\end{itemize}

\noindent
Assertions (1)-(5) imply the following assertion:
\begin{itemize}
\item[(6)]
$ P_{\cY_+} V_\alpha x = V_\alpha P_{\cY_+} x$
for all $\alpha \in F^+_d$ and $x \in \cU_+$
\end{itemize}
If in addition $\cY_0$ is also wandering for $V$ then 
(6) is equivalent to (1)-(5) and can be rewritten as
\begin{itemize}
\item[(6')]
$M\,S^\cU_\alpha = S^\cY_\alpha \,M \;$
for all $\alpha \in F^+_d$, \\
where 
$S^\cU$ and $S^\cY$ are the row shifts obtained by restricting $V$ to $\cU_+$ and $\cY_+$ and
$M = P_{\cY_+}|_{\cU_+}$
is the multi-Toeplitz operator introduced above.
\end{itemize}
\end{thm}

Let us describe the relative position of the embedded subspaces $\cU$ and $\cY$ characterized in Theorem \ref{thm:transfer} by saying that there is an {\it orthogonal $\cY$-past}. This terminology is suggested by (5) and some additional motivation for it is given at the end of this section.

\begin{proof}
$(1) \Leftrightarrow (2)$. In fact, 
\[
0 = K(0, \alpha) = j^*_0 V_\alpha i_0
\]
means exactly that $V_\alpha \cU_0$ is orthogonal to $\cY_0$ or,
equivalently, that $\cU_0$ is orthogonal to $V^*_\alpha \cY_0$. 

\noindent
$(2)\Rightarrow(3)$. If $(3)$ is not satisfied then there exists $y \in \cY_0$ and $\alpha \not= 0$ such that $P_{\cU_\alpha} y \not=0$. But then 
$P_{\cU_0} V^*_\alpha y \not=0$ contradicting $(2)$.

\noindent
$(3)\Rightarrow(4)$. Because for $k=1,\ldots,d$
\[
V_k \bigoplus_{\alpha\in F^+_d} \cU_\alpha
\subset \bigoplus_{\alpha \not= 0} \cU_\alpha
\perp \cH \oplus \cU_0
\]
we conclude from $\cY_0 \subset \cH \oplus \cU_0$ that
$\cU_+
\perp V^*_k \cY_0$, hence $V^*_k \cY_0 \subset \cH$.

\noindent
$(4) \Rightarrow (5)$ follows from the fact that $\cH$ is invariant for the $V^*_k$ and 

\noindent
$(5) \Rightarrow (2)$ is obvious. 

\noindent
$(3) \Rightarrow (6)$. It is elementary that
$P_{V_\alpha \cY_+} V_\alpha = V_\alpha P_{\cY_+}$ for all
$\alpha \in F^+_d$. To get (6), that is
$ P_{\cY_+} V_\alpha x = V_\alpha P_{\cY_+} x$ for all
$\alpha \in F^+_d$ and $x \in \cU_+$, it is therefore enough to
consider all vectors of the form $V_\beta y$ where $y \in \cY_0$ and $\beta \in F^+_d$ is a word which does not begin with $\alpha$ and
to show that such vectors are always orthogonal to $V_\alpha x$ where $x \in \cU_+$. By (3) we have 
$\cY_0 \subset \cH \oplus \cU_0$ which implies, because
$V_k \cH \subset \cH \oplus \cU_0$ for all $k=1,\ldots,d$, 
that $V_\beta y$ is contained in the span of $\cH$ and of all
$V_\gamma \cU_0$ where the word $\gamma \in F^+_d$  does not begin with $\alpha$. This is indeed orthogonal to $V_\alpha x$
because $\cU_0$ is wandering.

\noindent
Conversely we prove, under the additional assumption that $\cY_0$ is wandering, the implication $(6') \Rightarrow (2)$. 
If $(2)$ is not satisfied then there exists $u \in \cU_0$ and 
$\alpha \not= 0$ such that
$V_\alpha u$ is not orthogonal to $\cY_0$.
Hence
\[
P_{\cY_0}\, M\,S^\cU_\alpha u = P_{\cY_0} V_\alpha u \not=0\,.
\]
On the other hand, from $\cY_0$ wandering, we get
\[
P_{\cY_0} S^\cY_\alpha M u = 0
\]
and hence $M\,S^\cU_\alpha \not= S^\cY_\alpha\,M\,$.

\end{proof}

Note that if $\cY_0$ is not wandering then in general (6) does not imply (1)-(5),
in other words (6) may be true without $K$ being multi-analytic. Choose $\cY_0 = \hH$ for example. Though in this paper we are mainly interested in pairs of wandering subspaces it is very useful to observe that all the other implications in Theorem \ref{thm:transfer} hold more general. For example it can be convenient in applications to start with a bigger subspace $\cY_0$ and to restrict only later to a suitable wandering subspace. 
\\

Now consider the following operators:

\begin{eqnarray*}
A_k := V^*_k |_{\cH}   \colon \;\cH \rightarrow \cH,
&\quad&
B_k := V^*_k\, i_0       \colon \;\cU \rightarrow \cH, \quad\quad k=1,\ldots,d \\
C := j^*_0 |_{\cH} \colon \;\cH \rightarrow \cY,
&\quad&
D := \;j^*_0\, i_0     \colon \;\cU \rightarrow \cY\,. \\
\end{eqnarray*}

Note that the assumption that $\cU_0$ is wandering is needed to show that the $B_k$ map $\cU$ into $\cH$. If $K$ is multi-analytic then it is determined by these operators. In fact, it is elementary to check that
\[
K(\alpha, 0) = j^*_0 V^*_\alpha i_0 = 
\left\{ 
\begin{array}{cc}
D & {\mbox if}\; \alpha = 0 \\ 
C\, B_\alpha & {\mbox if}\; |\alpha| = 1 \\
C\, A_{\alpha_r} \ldots A_{\alpha_2}\,B_{\alpha_1} & {\mbox if}\;  \alpha = \alpha_1 \ldots \alpha_r,\, r = |\alpha| \ge 2
\end{array} 
\right.
\]

These formulas suggest an interpretation from the point of view of linear system theory.

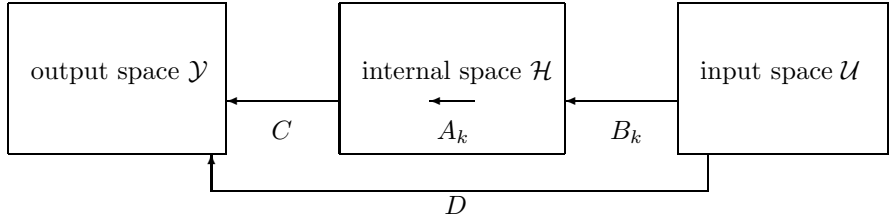
\setlength{\unitlength}{1cm}
\begin{picture}(15,3.5)

\put(0.0,1){\line(0,1){2}}
\put(0.0,1){\line(1,0){2.9}}
\put(2.9,1){\line(0,1){2}}
\put(0.0,3){\line(1,0){2.9}}

\put(4.4,1){\line(0,1){2}}
\put(4.4,1){\line(1,0){3}}
\put(7.4,1){\line(0,1){2}}
\put(4.4,3){\line(1,0){3}}

\put(8.9,1){\line(0,1){2}}
\put(8.9,1){\line(1,0){2.8}}
\put(11.7,1){\line(0,1){2}}
\put(8.9,3){\line(1,0){2.8}}

\put(4.4,1.7){\vector(-1,0){1.5}}
\put(6.2,1.7){\vector(-1,0){0.6}}
\put(8.9,1.7){\vector(-1,0){1.5}}

\put(0.3,2.0){output space $\cY$}
\put(9.2,2.0){input space $\cU$}
\put(4.7,2.0){internal space $\cH$}
\put(3.5,1.2){$C$}
\put(5.7,1.2){$A_k$}
\put(8.0,1.2){$B_k$}

\put(2.7,0.5){\line(1,0){6.6}}
\put(2.7,0.5){\vector(0,1){0.5}}
\put(9.3,0.5){\line(0,1){0.5}}
\put(5.8,0.2){$D$}

\end{picture}

In fact, if we interpret $u \in \cU$ as an input then we can think of $C A_\beta B_k u$ as a family of outputs originating from it, stored in suitable copies of $\cY$. Motivated by these observations we say, in the case of an orthogonal $\cY$-past, that the associated multi-analytic kernel $K$ (or the multi-analytic operator $M$ if available) is a {\it transfer function} (for the embedded spaces $\cU$ and $\cY$). 

We remark that the scheme is close to the formalism of Ball-Vinnikov in \cite{BV05b}, compare for example formula (3.3.2) therein. Essentially the same construction, but in a commutative-variable setting, appears in \cite{BSV05}. 
In the later section on Markov chains in this paper we
describe another reappearance of this structure which has been observed in \cite{Go09}. For the moment, to make our terminology even more plausible, let us consider the simplest case where $\cU_0$ and $\cY_0$ are both wandering and $d=1$ (i.e., $V$ is a single isometry). 
Let $H^2(\cU)$ resp. $H^2(\cY)$ denote the $\cU$-valued resp. $\cY$-valued Hardy space on the complex unit disc $\Dset$. For example a function in $H^2(\cU)$ has the form
\[
\Dset \ni z \mapsto \sum^\infty_{n=0} a_n z^n 
\quad {\mbox with}\; a_n \in \cU\,.
\]
There is a natural unitary from $\bigoplus^\infty_{n=0} \cU_n$
onto $H^2(\cU)$, taking the summands as coefficients (similar for $\cY$). It can be used to move operators from one Hilbert space to the other. For more details see for example \cite{FF90}, Chapter IX.
This allows us to summarize the previous discussions in this special case as follows.

\begin{cor} \label{cor:transfer}
If $\cU_0$ and $\cY_0$ are a pair of wandering subspaces (for an isometry $V$) with orthogonal $\cY$-past then 
$M := P_{\cY_+} |_{\cU_+}$ moved to the Hardy spaces becomes a contractive multiplication operator $M_\Theta$ with
\[
\Theta(z) = D + \sum^\infty_{n=1} C A^{n-1} B z^n 
= D + C (I_\cH - zA)^{-1} z B\,.
\]
Here $A:=A_1 = V^*|_\cH,\, B:=B_1 = V^* i_0$ and $\Theta \in H^\infty_1(\cU,\cY)$, the unit ball of the algebra of bounded analytic functions on $\Dset$ with values in $\cB(\cU,\cY)$, the bounded operators from $\cU$ to $\cY$.
\end{cor}

This means that in this case $M$ is an analytic operator in the sense of \cite{RR85} 
(except for the insignificant fact that it operates between different Hilbert spaces). 

In the general noncommutative case we can similarly encode all
the entries $K(\alpha,0)$ (as described above) into a
formal power series which fully describes a multi-analytic operator $M$.

\begin{cor} \label{cor:transfer2}
If $\cU_0$ is a wandering subspace for a row isometry $V = (V_1,\ldots,V_d)$ and $\cY_0$ is another subspace then, with indeterminates
$z_1,\ldots,z_d$ which are freely noncommuting among each other but commuting with the operators,
\[
\Theta(z_1,\ldots,z_d) 
:= \sum_{\alpha \in F^+_d} K(\alpha, 0)\, z^\alpha 
= D + C \sum^\infty_{r=1} (ZA)^{r-1} ZB
= D + C (I_\cH - ZA)^{-1} ZB
\]
where $Z = (z_1\,I_\cH, \ldots ,z_d\,I_\cH),\;
A = (A_1, \ldots, A_d)^T,\; B = (B_1, \ldots, B_d)^T$,
the transpose indicating that $A$ and $B$ should be interpreted as (operator-valued) column vectors.
Further $z^\alpha := z_{\alpha_n} \ldots z_{\alpha_1}$ if $\alpha = \alpha_1 \ldots
\alpha_n$ is a word of length $n$. 
\end{cor}

Such a formalism is explained in more detail and used systematically in \cite{BV05b}.
Using the language of system theory we have
all the relevant information in the socalled {\it system matrix}
\[
\Sigma =
\left( 
\begin{array}{cc}
A & B \\ 
C & D \\
\end{array} 
\right)\,.
\]

Let us put these results into the context of other work already done in operator theory and focus on the case $d=1$ again. We could have extended the isometry $V$ to a unitary $\tilde{V}$ on a larger Hilbert space. If we now define $\cY_k
= \tilde{V}^k \cY_0$ also for $k<0$ then it is natural to call $\bigoplus_{k<0} \cY_k$
the $\cY$-past. In this extended setting the
fact that we have orthogonal $\cY$-past ensures that $\tilde{V}$ is a coupling in the sense of \cite{FF90}, chapter VII.7, between the right shifts on the orthogonal spaces $\bigoplus^\infty_{k \ge 0} \cU_n$ and $\bigoplus_{k<0} \cY_k$.
Further our operator $M$ can now be interpreted as the contractive intertwining lifting of the zero intertwiner between the two shifts which is canonically associated to the coupling $\tilde{V}$. See \cite{FF90}, Chapter VII.8, for this construction. We don't go into this here, the book \cite{FF90} contains detailed discussions how analytic functions arise in the classification of such structures.

We remark that in the case $d > 1$ it is more complicated to develop
the analogue of such a `two-sided' setting but this has been worked out in \cite{BV05a,BV05b} within a theory of Haplitz kernels and  Cuntz weights. For the purposes of this paper it turns out that the simpler `one-sided' setting, as presented in this section and in particular in Theorem \ref{thm:transfer}, is sufficient. 

\section{Examples: Characteristic Functions}

The examples in this section are well known and the main emphasis is therefore to show that they fit naturally into the scheme developed in the previous section and that thinking about them in this way simplifies the constructions.
For further simplification we only work through the details of the case $d=1$, i.e., a single isometry $V: \hH \rightarrow \hH$.

Suppose that $\cU_0$ and $\cY_0$ are a pair of wandering subspaces with orthogonal $\cY$-past and with system matrix
\[
\Sigma =
\left( 
\begin{array}{cc}
A & B \\ 
C & D \\
\end{array} 
\right)
\colon
\cH \oplus \cU \rightarrow \cH \oplus \cY
\]
For the adjoint $\Sigma^*$ we obtain from the definition of $A,B,C,D$:
\[
\Sigma^* =
\left( 
\begin{array}{cc}
A^* & C^* \\ 
B^* & D^* \\
\end{array} 
\right)
\colon
h \oplus y \mapsto P_H \big[ Vh + j_0(y) \big]
\oplus i^*_0 P_{\cU_0} \big[ Vh + j_0(y) \big]
\]

\subsection{Example} 
Let us consider a special case of the previous setting where
$V \cH \perp j_0(\cY)$. Then $\Sigma^*$ is isometric, i.e., $\Sigma$ is a coisometry. 

\noindent
Conversely, for any Hilbert spaces $\cH,\, \cU,\, \cY$ let
$\Sigma =
\left( 
\begin{array}{cc}
A & B \\ 
C & D \\
\end{array}
\right)
\colon
\cH \oplus \cU \rightarrow \cH \oplus \cY$
be any coisometry. Now define the Hilbert space
$\hH := \cH \oplus \bigoplus^\infty_{n=0} \cU_n$ with the $\cU_n$
copies of $\cU$, the embeddings
\[
i_0(\cU) := \cU_0, \quad j_0 := (I_\cH \oplus i_0) \Sigma^* |_\cY
\]
and an isometry $V$ by $V |_\cH := (I_\cH \oplus i_0) \Sigma^* |_\cH$ and acting as a right shift on
$\cU_+ = \bigoplus^\infty_{n=0} \cU_n$. Then $\cU_0$ and $\cY_0$ are a pair of wandering subspaces with orthogonal $\cY$-past and with system matrix $\Sigma$. In fact, orthogonal $\cY$-past is clear from
$\cY_0 \subset \cH \oplus \cU_0$ and Theorem \ref{thm:transfer}
and then $\cY_0$ wandering is an immediate consequence of $V \cH \perp j_0(\cY)$ and the specific form of $V$. 

This situation occurs in the Sz.Nagy-Foias theory of characteristic functions for contractions. Let us sketch briefly how this fits in.
Let $T \in \cB(\cH)$ be a contraction. Then we have defect operators $D_T = \sqrt{I-T^*T}$ and $D_{T^*}$ with defect spaces
$\cD_T$ and $\cD_{T^*}$ defined as the closure of their ranges. 
The reader can easily check that the construction above applies with $\cU = \cD_T,\;\cY = \cD_{T^*}$ and with the unitary rotation matrix 
\[
\Sigma =
\left( 
\begin{array}{cc}
A & B \\ 
C & D \\
\end{array} 
\right)
=
\left( 
\begin{array}{cc}
T^* & D_T \\ 
D_{T^*} & -T \\
\end{array} 
\right)
\colon
\cH \oplus \cU \rightarrow \cH \oplus \cY
\]
Then $V$ is the minimal isometric dilation of $T$ and the transfer
function for the pair $\cU_0$ and $\cY_0$ given by
\[
\Theta(z) = -T + D_{T^*} (I_\cH - zT^*)^{-1} z\,D_T
\]
is nothing but the well-known Sz.Nagy-Foias {\it characteristic function} of $T$. In fact it is characteristic in the sense that it characterizes $T$ up to unitary equivalence only if $T$ is completely non-unitary (cf. \cite{SF70} or \cite{FF90}). So in the general case it may be better to refer to $\Theta$ as the transfer function associated to $T$. 

It is possible to handle the multi-variable case ($d >1$), first studied by Popescu in \cite{Po89b}, in a very similar way and the result, if expressed in the notation explained for Corollary \ref{cor:transfer2},
is very similar: The transfer function associated to a row contraction
$T = (T_1,\ldots,T_d) \colon \bigoplus^d_1 \cH \rightarrow \cH$ is 
\[
\Theta(z_1,\ldots,z_d) =  -T + D_{T^*} (I_\cH - Z T^*)^{-1} Z D_T 
\]
where $Z = (z_1\,I_\cH, \ldots ,z_d\,I_\cH)$. It is shown in  (\cite{Po89b}, 5.4) that $\Theta$ is characteristic in the sense of being a complete unitary invariant if $T$ is completely non-coisometric. It is further shown in (\cite{BV05b}, 5.3.3) that to get a complete unitary invariant in the class of completely non-unitary row contractions one can consider a characteristic pair $(\Theta, L)$ where $L$ is a Cuntz weight.

\subsection{Example}

But there are other possibilities to obtain a pair of wandering subspaces $\cU_0$ and $\cY_0$ with orthogonal $\cY$-past than the scheme explained in the previous example. We go back to the case $d=1$ and assume again that $\hH := \cH \oplus \bigoplus^\infty_{n=0} \cU_n$ and that
an isometry $V$ is given on $\hH$ which acts as a right shift on 
$\bigoplus^\infty_{n=0} \cU_n$. Now suppose further that $\cH_S$ is a subspace of $\cH$ which is $V^*$-invariant. Then for any subspace
$\cY_0$ satisfying
\[
\cY_0 \subset \overline{span}\{\cH_S, V \cH_S\} \ominus \cH_S
\]
it follows that $\cU_0$ and $\cY_0$ are a pair of wandering subspaces with orthogonal $\cY$-past. In fact, because 
$\cY_0 \perp \cH_S$ we have for $k \ge 1$ that $V^k \cY_0 \perp \cH_S$, but also $V^{k-1} \cY_0 \perp \cH_S$ so that 
$V^k \cY_0 \perp V \cH_S$. Hence $V^k \cY_0 \perp \cY_0$ for all $k \ge 1$, i.e., $\cY_0$ is a wandering subspace. Together with $V\,\cH \subset \cH \oplus \cU_0$ and 
Theorem \ref{thm:transfer} this establishes the claim.

This situation occurs in the theory of characteristic functions for liftings (cf. \cite{DG11}). As this is less well known than the Sz.Nagy-Foias theory in the previous example and the presentation in \cite{DG11} 
gives the general case $d \ge 1$ using a different
approach and a different notation we think it is instructive to work out explicitly some details of this transfer function in the case $d=1$ with the methods of this paper.

As in the previous subsection  
let $T \in \cB(\cH)$ be a contraction, $\cU := \cD_T,
\; \hH := \cH \oplus \bigoplus^\infty_{n=0} \cU_n,\;i_0(\cU) = \cU_0$, $\,V$ the minimal isometric dilation and we still have
$A = V^* |_\cH = T^*$ and $B = V^* i_0 = D_T$. But now suppose that $\cH = \cH_S \oplus \cH_R$ such that $\cH_S$ is invariant for $T^*$, in other words $T$ is a block matrix
\[
T =
\left( 
\begin{array}{cc}
S & 0 \\ 
Q & R \\
\end{array} 
\right)
\]
with respect to $\cH = \cH_S \oplus \cH_R$. We also say that 
$T \in \cB(\cH)$ is a {\it lifting} of $S \in \cB(\cH_S)$. Then $V$ is also an isometric dilation of $S$, i.e., $P_{\cH_S} V^n |_{\cH_S} = S^n$ for all $n \in \Nset$,
and it restricts to the minimal isometric dilation $V_S$ of $S$ on a reducing subspace. The subspace $\cH_S$ is invariant for $V^*$ and we obtain a situation as described in the beginning of this subsection by putting $\cY := \cD_{S}$ and
for $h_S \in \cH_S$
\[
j_0(D_{S} h_S) := (V_{S}-S) h_S = (V-S) h_S
= Q h_S \oplus i_0(D_T h_S) \in \cH_R \oplus \cU_0\,.
\]
Hence we have orthogonal $\cY$-past and $\cU_0$ and $\cY_0$ are both wandering. 

It is well known about contractive liftings such as $T$ that we always have
\[
Q = D_{R^*}\, \gamma^*\, D_{S}: \cH_S \rightarrow \cH_R
\]
with a contraction $\gamma: \cD_{R^*} \rightarrow \cD_{S}$
(cf. \cite{FF90}, Chapter IV, Lemma 2.1). 
We conclude that
\[
C = j^*_0 |_\cH = \gamma\, D_{R^*}\, P_{\cH_R}\,.
\]
To compute $D = j^*_0 i_0$ more explicitly note that for
$h_S \in \cH_S,\,h_R \in \cH_R$
\[
j^*_0 V h_S = j^*_0 \big( S h_S \oplus j_0(D_{S} h_S) 
= D_{S} h_S,
\quad\quad j^*_0 V h_R = 0,
\]
[the latter because $V \cH_R \perp \overline{span}\{V \cH_S, \cH_S\}
\supset j_0(\cD_{S}$)].

With $\cH \ni h = h_S \oplus h_R \in \cH_S \oplus \cH_R$ we can compute $D$ as follows:
\begin{eqnarray*}
D (D_T h) &=& j^*_0 \big( (V-T)h \big) \\
&=& j^*_0 Vh - j^*_0 Th 
= j^*_0 (V h_S + V h_R) - CT h \\
&=& D_{S} h_S - \gamma D_{R^*} (T h_S + R h_R) \\
&=&(D_{S} - \gamma D_{R^*} Q) h_S 
- \gamma D_{R^*} R h_R \\
&=&(D_{S} - \gamma D_{R^*} Q) h_S 
- \gamma R D_{R} h_R \\
\end{eqnarray*}
(using $D_{R^*} R = R D_{R}$ in the last line). Hence we get a transfer function
\begin{eqnarray*}
\Theta(z) = D + \sum_{n\ge 1} \gamma D_{R^*} P_{\cH_R} (z T^*)^{n-1} z D_T  
= D + \gamma D_{R^*} P_{\cH_R} (I_\cH - z T^*)^{-1} z D_T \\
= D + \sum_{n\ge 1} \gamma D_{R^*} (z R^*)^{n-1} P_{\cH_R} z D_T
= D + \gamma D_{R^*} (I_{\cH_R} - z R^*)^{-1} P_{\cH_R} z D_T
\end{eqnarray*}
The domain of $\Theta(z)$ (for each $z$) is $\cU = \cD_T$. We gain additional insights by evaluating
$\Theta(z)$ on
$D_T h_R = D_R h_R$ with $h_R \in \cH_R$ and on $D_T h_S$ with $h_S \in \cH_S$. 
\begin{eqnarray*}
\Theta(z)(D_T h_R) &=& D (D_T h_R)
+ \sum_{n\ge 1} \gamma D_{R^*} (z R^*)^{n-1} P_{\cH_R} z D_T (D_T h_R) \\
&=& \gamma \big[ -R + D_{R^*} (I - z R^*)^{-1} z D_R \big] (D_R h_R)
\end{eqnarray*} 
which shows that
$\Theta(z)$ restricted to $D_T \cH_R = D_{R} \cH_R$ is nothing but $\gamma$ times the transfer function associated to $R$ in the sense of Sz.Nagy and Foias, as discussed in the previous subsection. Its presence can be explained by the fact that $V$ restricted to $\hH \ominus \cH_S$ also provides an isometric dilation of $R$. For the other restriction 
$\Theta(z) |_{D_T \cH_S}$
we find, using 
$P_{\cH_R} z D^2_T h_S = P_{\cH_R} z (I - T^*T) h_S = - z R^* Q \, h_S$,
\begin{eqnarray*}
\Theta(z) (D_T h_S) 
&=& D (D_T h_S)
+ \sum_{n\ge 1} \gamma D_{R^*} (z R^*)^{n-1} P_{\cH_R} z D_T (D_T h_S) \\
&=& \big[ D_{S} - \gamma D_{R^*} Q
- \sum_{n\ge 1} \gamma D_{R^*} (z R^*)^n Q \big] (h_S) \\
&=& \big[ D_{S} - \gamma D_{R^*}\,
\sum_{n\ge 0} (z R^*)^n Q \big] (h_S) \\
&=& \big[ I - \gamma D_{R^*} (I-z R^*)^{-1} D_{R^*} \gamma^* \big] D_{S} h_S \\
\end{eqnarray*}

Again the multi-variable case ($d>1$) can be handled similarly and yields similar results. Here we
investigate a row contraction $T = (T_1,\ldots,T_d)$ which is a lifting in the sense that
\[
T_k =
\left( 
\begin{array}{cc}
S_k & 0 \\ 
Q_k & R_k \\
\end{array} 
\right)
\]
for all $k=1,\ldots,d$. All the formulas for transfer functions derived above have been written in a form which makes sense and which is still valid for the multi-variable case if we simply replace the variable $z$ by $Z = (z_1 I, \ldots, z_d I)$ (as in Corollary
\ref{cor:transfer2}) and the vectors $h_S,\,h_R$ by
$d$-tuples of vectors. 

These transfer functions have been introduced in \cite{DG11};
see Section 4 therein, in particular formulas (4.6)
and (4.5), for an alternative approach to the facts sketched above. 
Among other things it is further investigated in \cite{DG11} in which cases such transfer functions are characteristic for the lifting, i.e., characterize the lifting $T$ given $S$ up to unitary equivalence. 

\section{Examples: Noncommutative Markov Chains} 

There is another way how transfer functions as described in Section 1 appear in applications, namely in the theory of noncommutative Markov chains. This has been observed in \cite{Go09} and to work out a common framework in order to facilitate the discussion of similarities has been a major motivation for this paper.

We quickly review the setting of \cite{Go09} as far as it is needed to make our point, referring to that paper for more details. An {\it interaction}
is defined as a unitary
\[
U \colon \cH \otimes \cK \rightarrow \cH \otimes \cP
\]
where $\cH, \cK, \cP$ are Hilbert spaces. In quantum physics it is common to describe the aggregation of different parts by a tensor product of Hilbert spaces and in this case we may think of $U$ as one step of  a discretized interacting dynamics. (For such an interpretation we may take $\cK=\cP$ and think of $\cK$ and $\cP$ as describing the same part before and after the interaction. But mathematically it is more transparent to treat them as two distinguishable  spaces.) If $\cH$ represents a fixed quantum system, say an atom, and interactions take place with a wave passing by, say a light beam, then it is natural, at least as a toy model, to represent repeated interactions ($n$ steps) by
\[
U(n) := U_n \ldots U_2 U_1 \colon 
\cH \otimes \bigotimes^n_{\ell=1} \cK_\ell
\mapsto
\cH \otimes \bigotimes^n_{\ell=1} \cP_\ell
\]
where the $\cK_\ell$ (resp. $\cP_\ell$) are copies of $\cK$ (resp. $\cP$), and $U_\ell$ acts as $U$ from $\cH \otimes \cK_\ell$ to $\cH \otimes \cP_\ell$, identical at the other parts. 

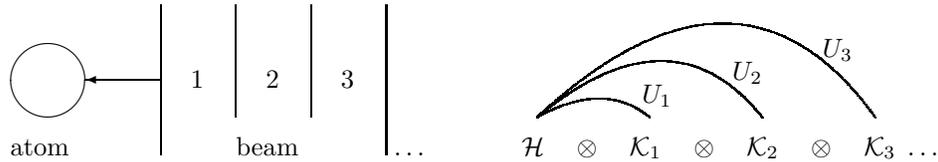
\setlength{\unitlength}{1cm}
\begin{picture}(15,3.5)
\put(0.5,2){\circle{1}} 
\put(2.0,2){\vector(-1,0){1}}
\put(2.0,1){\line(0,1){2}}
\put(3.0,1.5){\line(0,1){1.5}}
\put(4.0,1.5){\line(0,1){1.5}}
\put(5.0,1){\line(0,1){2}}
\put(5.1,1){\ldots}
\put(2.4,1.9){$1$}
\put(3.4,1.9){$2$}
\put(4.4,1.9){$3$}
\put(0,1){atom}
\put(3.0,1){beam}
\put(6.8,1){$\cH \quad \otimes \quad \cK_1 \quad \otimes \quad \cK_2 \quad \otimes \quad \cK_3 \; \ldots$}
\qbezier(7,1.5)(7.9,2)(8.5,1.5)
\qbezier(7,1.5)(8.8,3)(10,1.5)
\qbezier(7,1.5)(9.7,4)(11.5,1.5)
\put(8.4,1.7){$U_1$}
\put(9.6,2){$U_2$}
\put(10.8,2.3){$U_3$}
\end{picture}

Choosing unit vectors $\Omega_\cK \in \cK$ and $\Omega_\cP \in \cP$ we can also form infinite tensor products along these unit vectors and obtain $U(n)$ for every $n \in \Nset$ on a common Hilbert space. 
Such a toy model of quantum repeated interactions can mathematically be thought of as a noncommutative Markov chain. We refer to \cite{Go09} for some motivation for this terminology by analogies with classical Markov chains.

It is proved in  \cite{Go09} (in a slightly different language) that if we have another unit vector $\Omega_\cH \in \cH$ such that 
$U (\Omega_\cH \otimes \Omega_\cK )
= \Omega_\cH \otimes \Omega_\cP$
(we call these unit vectors {\it vacuum vectors} in this case)
then we obtain a pair of wandering subspaces $\cU_0$ and $\cY_0$
with orthogonal $\cY$-past for a row isometry $V$, notation consistent with Section 1, as follows:
\[
\hH := \cH \otimes \bigotimes^\infty_{\ell=1} \cK_\ell \quad
\supset \cH  \otimes \bigotimes^\infty_{\ell=1} \Omega_{\cK_\ell}
\simeq \cH
\]
i.e., the latter subspace of $\hH$ is identified with $\cH$. The row isometry $V$
on $\hH$ is of the form 
\[
V := (V_1, \ldots, V_d), \quad d = dim\, \cP,
\]
where $dim\, \cP$ is the number of elements in an orthonormal basis
of $\cP$. Let $(\epsilon_k)^d_{k=1}$ be such an orthonormal basis of $\cP = \cP_1$, fixed from now on. 

\noindent
Then for $\xi \in \cH$ and $\eta \in \bigotimes^\infty_{\ell=1} \cK_\ell$
\[
V_k \big( \xi \otimes \eta \big) := U^*_1(\xi \otimes \epsilon_k \otimes \eta)
\in (\cH \otimes \cK_1) \otimes \bigotimes^\infty_{\ell=2} \cK_\ell
\]
Note that $\eta$ is shifted to the right in the tensor product and appears as $\eta \in \bigotimes^\infty_{\ell=2} \cK_\ell$ on the right hand side. It is immediate that $V$ is a row isometry. 
To get used to this definition the reader is invited to verify the formula
\[
V_\alpha \big( \xi \otimes \eta \big)
= U(r)^* (\xi \otimes \epsilon_{\alpha_1} \otimes \ldots \otimes \epsilon_{\alpha_r} \otimes \eta)
\in (\cH \otimes \cK_1 \otimes \ldots \cK_r) 
\otimes \bigotimes^\infty_{\ell=r+1} \cK_\ell
\]
where $\alpha = \alpha_1 \ldots \alpha_r \in F^+_d$ with
$|\alpha| = r$
and on the right hand side
$\eta$ now appears as  $\eta \in \bigotimes^\infty_{\ell=r+1} \cK_\ell$. It becomes clear that the properties of the repeated interaction are encoded into properties of the row isometry $V$.
\\

Finally we define the pair of embedded subspaces:
\begin{eqnarray*}
\cU &:=&  \cH \otimes \big(\Omega_{\cK}\big)^\perp 
\subset \cH \otimes \cK \,,\\
\cU_0=i_0(\cU)  &:=& \cH \otimes \big(\Omega_{\cK_1}\big)^\perp 
\otimes  \bigotimes^\infty_{\ell=2} \Omega_{\cK_\ell}\,,
\end{eqnarray*}
\begin{eqnarray*}
\cY &:=& \big(\Omega_{\cP} \big)^\perp
\subset \cP\,, \\
\cY_0=j_0(\cY) &:=& U^*_1 \;\big( 
\Omega_\cH \otimes (\Omega_{\cP_1} )^\perp
\otimes  \bigotimes^\infty_{\ell=2} \Omega_{\cK_\ell} \big)\,.
\end{eqnarray*}

From the specific form of the isometries $V_k$ it is easy to check that $\cU_0$ is wandering and that $\hH = \cH \oplus \cU_+$.
The proof that $\cY_0$ is wandering can be found in \cite{Go09} or deduced from Proposition \ref{prop:wandering} below (which covers a more general situation).
From Theorem \ref{thm:transfer} we have an associated transfer function which can be made explicit as a multi-analytic kernel $K$ or as a (contractive) multi-analytic operator $M$. It may be called a transfer function of the noncommutative Markov process.
With $h \oplus u \in \cH \oplus \cU = \cH \otimes \cK$ (here we identify $\cH$ with $\cH \otimes \Omega_\cK$) we find the operators $A_k, B_k, C, D$ appearing in the system matrix 
$\Sigma$ to be related to the interaction $U$ by
\begin{eqnarray*}
U (h \oplus u) &=& \sum^d_{k=1} \big(A_k h + B_k u \big) \otimes \epsilon_k  \;\in \cH \otimes \cP \\
P_{\Omega_\cH \otimes \cY}\, U (h \oplus u) &=& C h + D u \;\in \cY
\end{eqnarray*}
(where we have to identify $\Omega_\cH \otimes \cY$ and $\cY$ for the last equation).
It is further discussed in \cite{Go09} how for models in quantum physics these operators and the coefficients of the transfer function built from them can be interpreted, and it is shown that the transfer function can be used to study questions about observability and about scattering theory (outgoing Cuntz scattering systems \cite{BV05b} and scattering theory for Markov chains \cite{KM00}).  

Let us finally indicate that the additional ideas introduced in this paper provide a
flexible setting for generalizations. Let us consider the situation above but without assuming the existence of vacuum vectors.
With $\Omega_\cK \in \cK$ being an arbitrary unit vector we can easily check that $\cU_0$ as defined above is still a wandering subspace for $V$. Hence for an arbitrary subspace $\cY_0$ of 
\[
\cH \oplus \cU_0 = \cH \otimes \cK_1 \otimes \bigotimes^\infty_{\ell=2} \Omega_{\cK_\ell}
\]
we conclude, by Theorem \ref{thm:transfer}, that we have orthogonal $\cY$-past and there exists a corresponding transfer function corresponding to a multi-analytic kernel $K$. When is $\cY_0$ wandering? A sufficient criterion generalizing the situation with vacuum vectors is provided by the following

\begin{prop} \label{prop:wandering}
Let $\cH_S$ be a subspace of $\cH$ such that $U \big( \cH_S \otimes \Omega_{\cK} \big) \subset
\cH_S \otimes \cP$. Then any subspace 
\[
\cY_0 \subset U^*_1( \cH_S \otimes \cP_1) \ominus (\cH_S \otimes \Omega_{\cK_1}) 
\]
is wandering.
\\
(Here we adapt the convention to omit a tensoring with
$\bigotimes^\infty_{\ell=2} \Omega_{\cK_\ell}$ in the notation.)
\end{prop}

\begin{proof}
Let $\zeta \in \hH$ be any vector orthogonal to  
$\cH_S \otimes \Omega_{\cK_1}$.
Our first observation is that for all $k=1,\ldots,d$ the vectors
$V_k \zeta$ are orthogonal to 
$\cH_S \otimes \Omega_{\cK_1}$ too.
In fact, we can write $\zeta = \zeta_1 \oplus \zeta_2$ 
where 
$\zeta_1 = \xi_0 \otimes \eta$ with $\xi_0 \in \cH_S$ and with  $\eta \in \bigotimes^\infty_{\ell=1} \cK_\ell$ orthogonal to
$\bigotimes^\infty_{\ell=1} \Omega_{\cK_\ell}$
and
$\zeta_2 \in (\cH \ominus \cH_S) \otimes 
\bigotimes^\infty_{\ell=1} \cK_\ell$.
Using the specific form of $V_k$ it follows immediately that $V_k \zeta_1$ is orthogonal to  
$\cH_S \otimes \Omega_{\cK_1}$
and the same is also true for $V_k \zeta_2$ taking into account the assumption
$U \big( \cH_S \otimes \Omega_{\cK} \big) \subset
\cH_S \otimes \cP$, in the form:
$U^*_1 \big((\cH \ominus \cH_S) \otimes \cP_1 \big)$ is orthogonal to
$\cH_S \otimes \Omega_{\cK_1}$.

The second observation is that for all $k=1,\ldots,d$ the vectors
$V_k \zeta$ are orthogonal to $U^*_1 (\cH_S \otimes \cP_1)$.
As $\zeta$ can be approximated by a finite sum
$\sum_j \xi_j \otimes \eta_j$ with $\xi_j \in \cH$ and
$\eta_j \in \bigotimes^\infty_{\ell=1} \cK_\ell$ we may assume for simplicity that $\zeta$ is of this form. 
But then
$\sum_j \xi_j \otimes \epsilon_k \otimes \eta_j$
is orthogonal to
$\cH_S \otimes \cP_1$
and  now an application of $U^*_1$ gives us the result.

Applying these observations repeatedly to elements of $\cY_0$ we conclude that
$\cY_0$ is orthogonal to $V_\alpha \cY_0$ for all $\alpha \not= 0$. This implies that $\cY_0$ is wandering.
\end{proof}

\subsection*{Acknowledgment}
During my visit to Mumbai in $2010$ results closely related to Proposition \ref{prop:wandering} were communicated to me by Santanu Dey. These were in the back of my mind when I worked out a version fitting into the setting of this paper. Further I want to thank the referee for constructive comments leading to a better presentation and for pointing out additional links to the existing literature.
\end{document}